\documentclass[12pt]{amsart}
\usepackage[T1]{fontenc}
\usepackage[utf8]{inputenc}
\usepackage{lmodern}
\usepackage[margin=1.02in]{geometry}
\usepackage{amsmath,amssymb,amsthm,mathtools}
\usepackage{enumitem}
\usepackage{graphicx}
\usepackage[colorlinks=true,linkcolor=blue,citecolor=blue,urlcolor=blue]{hyperref}
\usepackage{xcolor}
\hypersetup{
    colorlinks,%
    citecolor=black,%
    filecolor=black,%
    linkcolor=black,%
    urlcolor=black
}
\usepackage{tikz}
\usepackage{float}
\usetikzlibrary{arrows.meta,calc}

\numberwithin{equation}{section}
\numberwithin{figure}{section}

\newtheorem{theorem}{Theorem}[section]

\newtheorem{proposition}[theorem]{Proposition}
\newtheorem{lemma}[theorem]{Lemma}
\newtheorem{remark}[theorem]{Remark}

\newcommand{\sech}{\operatorname{sech}}

\title[Non-radial solutions of  the critical H\'enon equation]{Existence of non-radial entire solutions for the H\'enon equation beyond even exponents}

\author{Qinfeng Jiang}
\address{School of Mathematical Sciences,
Beijing Normal University\\
Beijing 100875, China}
\email{202531130031@mail.bnu.edu.cn}

\author{Jingang Xiong}
\address{School of Mathematical Sciences, Laboratory of Mathematics and Complex Systems, MOE\\
Beijing Normal University, Beijing 100875, China\\
Center for Basic Mathematics, Institute for Advanced Study\\
Beijing Normal University, Beijing 100875, China}
\email{jx@bnu.edu.cn}

 \thanks{J. Xiong was partially supported by NSFC grant 12325104.}


\begin{document}

\begin{abstract}
This paper is concerned with the existence of non-radial positive
classical solutions for the critical H\'enon equation
\[
-\Delta u=|x|^\alpha u^{\frac{N+2+2\alpha}{N-2}}
\qquad \text{in }\mathbb R^N,
\]
where \(\alpha>0\) and \(N\ge3\), satisfying the Newtonian-type decay condition at infinity.

Gladiali, Grossi and Neves (2013) proved existence for the discrete sequence $\alpha_k=2(k-1)$, $k\in\mathbb N$, and conjectured that non-radial solutions may exist only at these special values. We disprove this conjecture by establishing existence for a continuum of exponents near each \(\alpha_k\): for every even $k>\frac{N-2}{2}$, non-radial solutions persist for parameters \(\alpha\) close to, and different from, \(\alpha_k\).

We recast the problem as a semilinear elliptic equation with Sobolev-supercritical exponent on the cylinder via the Emden--Fowler change of variables. Our argument is formulated directly on the cylindrical domain, thereby streamlining the characterization of the kernel of the linearized operator via P\"oschl--Teller spectral theory, avoiding the ball-exhaustion technique employed in the original work, and allowing us to compute the bifurcation slope and verify the non-verticality condition.

\end{abstract}
\maketitle

\section{Introduction}

The H\'enon equation
\begin{equation}\label{eq:Henon}
-\Delta u=|x|^\alpha u^p, \quad u\ge 0,
\end{equation}
with \(x\in\mathbb R^N\), \(N\ge 3\), and \(p>1\), was introduced by H\'enon
in the study of stellar dynamics as a radially weighted variant of the
Lane--Emden equation.  In this model, the unknown \(u\) represents a density or
potential-type quantity, while the factor \(|x|^\alpha\) describes a radial
inhomogeneity of the medium or of the source term.  The parameter
\(\alpha>0\) measures the strength of this radial weight: the nonlinearity is
weaker near the origin and stronger away from it.  Although the original
physical motivation is three-dimensional, the equation has been extensively
studied for general \(N\ge3\), where the weight \(|x|^\alpha\) leads to rich
variational, symmetry-breaking, and bifurcation phenomena.

In this paper, we consider the critical H\'enon equation, namely
\begin{equation}\label{eq:paramter}
\alpha \ge 0, \qquad p=p_\alpha:=\frac{N+2+2\alpha}{N-2}.
\end{equation}
This critical equation is invariant under the scaling
\[
u(x) \mapsto \lambda^{\frac{N-2}{2}} u(\lambda x), \qquad \lambda>0,
\]
and admits an explicit positive radial solution
\begin{equation}\label{eq:Ualpha}
U_\alpha(x)=c(N,\alpha) \frac{1}{(1+|x|^{2+\alpha})^{\frac{N-2}{2+\alpha}}},
\end{equation}
where
\[
c(N,\alpha)=\big[(N+\alpha)(N-2)\big]^{\frac{1}{p_\alpha-1}}.
\]
When $\alpha=0$, the classical result of Caffarelli, Gidas and Spruck \cite{CGS} establishes that, up to translation and scaling, $U_\alpha$ is the unique positive solution. For $\alpha>0$, however, the presence of the weight $|x|^\alpha$ breaks the translation invariance of the equation. Moreover, from the perspective of non-radial perturbations, the equation becomes Sobolev supercritical.

In 2013, Gladiali, Grossi and Neves \cite{GGN} proved the existence of non-radial positive solutions to the critical H\'enon equation \eqref{eq:Henon} satisfying the $|x|^{2-N}$ decay condition at infinity, for the discrete sequence of exponents
\begin{equation}\label{eq:alphak}
\alpha_k=2(k-1),\quad k\in\mathbb N.
\end{equation}
They further conjectured that non-radial solutions may exist only at these special parameter values. Our main theorem disproves this conjecture:

\begin{theorem}\label{thm:main}
Let $N\ge 3$ and let $k>\frac{N-2}{2}$ be an even integer. 
Then there exist \(\varepsilon>0\) and a continuous family
\(\{u_\alpha\}_{|\alpha-\alpha_k|<\varepsilon}\) of positive classical solutions to
\eqref{eq:Henon} with \(p=p_\alpha\), bifurcating from
\((\alpha_k,U_{\alpha_k})\), such that \(u_{\alpha_k}=U_{\alpha_k}\), and
\(u_\alpha\) is non-radial for \(0<|\alpha-\alpha_k|<\varepsilon\).
 Moreover, each $u_\alpha$ is $O(N-1)\times O(1)$-invariant and satisfies the decay estimate
\[
u_\alpha(x)=O(|x|^{2-N}) \qquad \text{as } |x|\to\infty.
\]
\end{theorem}

A few remarks are as follows.

(i) Combining Theorem~\ref{thm:main} with the global bifurcation theorem from
\cite{GGN} (see also Theorem~\ref{thm:global-cylinder} and Remark~\ref{gbilm} below), we
deduce the existence of a global continuum containing the aforementioned local
bifurcation branch.  More precisely, Proposition~\ref{prop:non-vert} shows
that this local branch is non-vertical with respect to the parameter \(\alpha\);
near \((\alpha_k,U_{\alpha_k})\), it satisfies
\[
\alpha(s)=\alpha_k+\beta_{N,k}s+O(s^2),
\qquad
\beta_{N,k}<0.
\]
Thus the branch actually leaves the hyperplane \(\alpha=\alpha_k\).  The
global continuum containing this branch is either unbounded, meets the critical
endpoint \(\alpha=0\), or connects back to another even bifurcation point
\((\alpha_j,U_{\alpha_j})\) within the same symmetry class.

(ii) Two heuristic arguments were put forward in \cite{GGN} to support their
conjecture.  The first one, \cite[Proposition 1.7]{GGN}, constructs an explicit
branch of non-radial solutions for \(\alpha=2\) and even \(N\), with solutions
depending only on \((|x'|,|x''|)\) under the decomposition
\(\mathbb R^N=\mathbb R^{N/2}\times\mathbb R^{N/2}\).  This branch exploits the
strong symmetry reduction under \(O(N/2)\times O(N/2)\), which effectively
lowers the dimension and brings the problem into a subcritical or critical
regime where bifurcation can be detected explicitly.  In our setting, however,
the non-radial solutions we obtain are invariant under the smaller group
\(O(N-1)\times O(1)\), which does not reduce the dimension in such a drastic
way.  The second concerns the Liouville-type equation in dimension two,
\[
-\Delta u = 2(\alpha+2)^2 |x|^\alpha e^u \quad \text{in } \mathbb R^2,
\]
with finite total curvature, which admits non-radial solutions if and only if
\(\alpha\) is an even integer, as classified in Prajapat--Tarantello
\cite{PT}.  In this two-dimensional setting, the exponential nonlinearity
\(e^u\) is always critical in the sense of the Moser--Trudinger inequality.
These examples thus reflect special critical or dimension-reducing structures,
rather than the genuinely Sobolev supercritical feature of the
higher-dimensional H\'enon equation.  Our result is consistent with the broader
theory of supercritical semilinear elliptic equations; see, e.g.,
Badiale--Serra \cite{BadialeSerra}, Figueroa--Neves \cite{FigueroaNeves},
Boscaggin--Colasuonno--Noris--Weth
\cite{BoscagginColasuonnoNorisWeth2023}, Cowan--Moameni
\cite{CowanMoameni2024} and references therein.

(iii) On the technical front, our proof of the main theorem follows a distinct
strategy compared to the argument in \cite{GGN}.  We first transform the
critical H\'enon equation on \(\mathbb R^N\) into a semilinear elliptic equation
posed on the cylinder
\(\mathcal C = \mathbb R \times \mathbb S^{N-1}\) via the Emden--Fowler change
of variables.  As illustrated below in \eqref{eq:cylinder-general}, this
transformation absorbs the weight term, yielding a standard semilinear elliptic
equation with supercritical nonlinearity.  In these cylindrical coordinates,
the radial solution \(U_\alpha\) reduces to a one-dimensional homoclinic
profile \(W_{p_\alpha}(t)\), and the associated linearized operator takes the
form of a Schr\"odinger operator with a P\"oschl--Teller potential.  This
structural simplification enables an explicit characterization of the
operator's kernel via classical spectral theory and streamlines the original
proof of \cite[Theorem 1.3]{GGN}.  Second, we develop a
Crandall--Rabinowitz bifurcation framework within suitably weighted H\"older
spaces invariant under the \(O(N-1)\times O(1)\) symmetry group.  Our
construction relies on linear elliptic operator theory on complete manifolds
established by Lockhart and McOwen \cite{LM}; see also Pacard \cite{P}.  The
slope computation in Proposition~\ref{prop:non-vert} is the key new ingredient
which shows that the local bifurcation branch leaves the even parameter
\(\alpha_k\).  Last but not least, we supply a concise alternative proof of the
global bifurcation result originally due to Gladiali, Grossi, and Neves.  Our
argument is formulated directly on the cylindrical domain, eliminating the
ball-exhaustion technique employed in their original treatment.

(iv) Very recently, Dai, Duan, Gui and Li \cite{DaiDuanGuiLi2026} extended the
results of Gladiali, Grossi and Neves \cite{GGN} to the critical quasilinear
H\'enon equation involving the \(p\)-Laplacian.  It would be interesting to
investigate whether our bifurcation results persist in that quasilinear setting.
Another natural direction concerns singular solutions.  Radial singular
solutions of the critical H\'enon equation are described, in cylindrical
coordinates, by Fowler or Delaunay-type profiles.  One may therefore ask whether
these radial singular profiles can bifurcate into non-radial singular solutions.
In the classical case \(\alpha=0\), corresponding to the Yamabe equation, such a
non-radial periodic bifurcation does not occur in the isolated-singularity
setting; see Caffarelli--Gidas--Spruck
\cite{CGS}. See also Marques \cite{Marques2008}, Xiong--Zhang
\cite{XiongZhang2022}, and Han--Xiong--Zhang \cite{HanXiongZhang2023} for related studies of isolated singularities.

Before concluding this introduction, we include the following schematic picture to illustrate the distinction between branches that remain at the even parameters and the non-vertical branches constructed in Theorem~\ref{thm:main}.

\begin{figure}[H]
\centering
\includegraphics[width=0.6\textwidth]{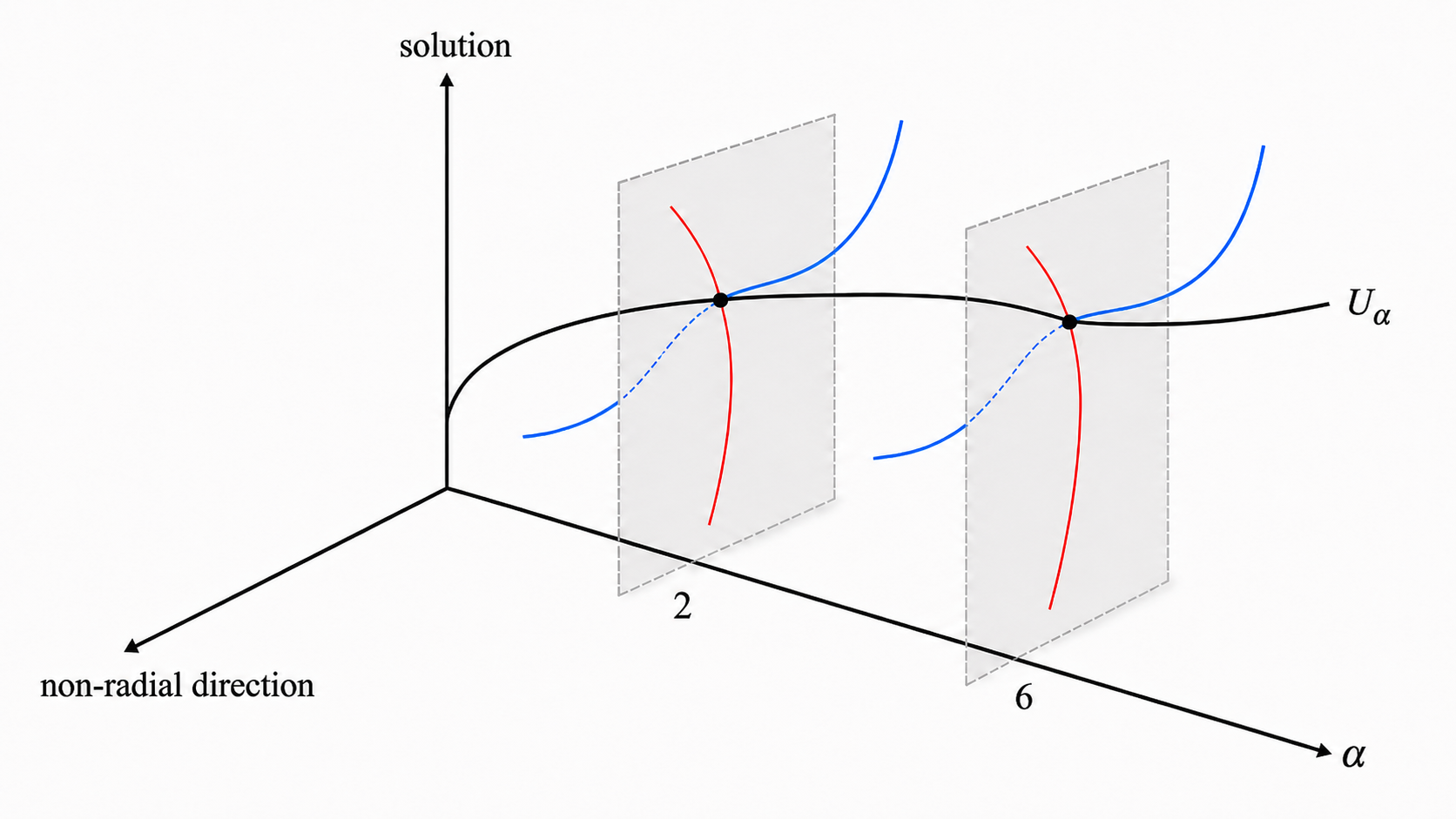}
\caption{Schematic bifurcation picture in the case \(N=3\). The dashed planes correspond to even parameters. 
The black radial branch intersects each plane at the bifurcation point. 
The red branches remain inside the corresponding plane, while the blue branches 
pass through the same point and cross the plane.}
\label{fig:actual}
\end{figure}

The paper is organized as follows. In Section~\ref{sec:cylinder} we derive the
cylinder equation and reprove the kernel characterization in
Theorem~\ref{thm:kernel}. Section~\ref{sec:local} sets up the weighted function spaces, constructs the local Crandall--Rabinowitz branch, and computes its parameter slope. This gives the non-verticality of the branch and completes the proof of Theorem~\ref{thm:main}. Finally, in Section~\ref{sec:global} we give a direct cylindrical proof of the global bifurcation theorem of \cite[Theorem 1.6]{GGN}.

\section{Cylinder formulation and the linearized operator}\label{sec:cylinder}

Let
\begin{equation}
    \label{clvar}
    r=|x|=e^t,
    \qquad
    \theta=\frac{x}{|x|}\in\mathbb S^{N-1},
    \qquad
    \mathcal C=\mathbb R\times\mathbb S^{N-1}.
\end{equation}
For any positive solution of \eqref{eq:Henon}, setting
\begin{equation}
    \label{clfun}
    w(t,\theta)= r^{\frac{N-2}{2}} u(x)
\end{equation}
yields
\begin{equation}\label{eq:cylinder-general}
    \mathcal P w:=-w_{tt}-\Delta_{\mathbb S^{N-1}}w+m^2w=w^{p_{\alpha}} \quad \text{on } \mathcal C,
\end{equation}
where \(\Delta_{\mathbb S^{N-1}}\) denotes the Laplace--Beltrami operator on \(\mathbb S^{N-1}\), and
\begin{equation}\label{eq:mdef}
    m=\frac{N-2}{2}.
\end{equation}
Note that \eqref{eq:cylinder-general} is Sobolev supercritical when \(\alpha>0\), since \(p_\alpha>\frac{N+2}{N-2}\).

By the standard ODE classification of positive homoclinic solutions (see, e.g., \cite{Perko}), these solutions are given, up to translation, by
\begin{equation}\label{eq:Wp-general}
\begin{split}
W_{p_{\alpha}}(t)
&= \left( \frac{m^2(p_{\alpha}+1)}{2} \right)^{\frac{1}{p_{\alpha}-1}}
\operatorname{sech}^{\frac{2}{p_{\alpha}-1}} \left( \frac{m(p_{\alpha}-1)}{2}t \right), \\
&= \bigl( m(m+q_{\alpha}) \bigr)^{\frac{m}{2q_{\alpha}}}
\operatorname{sech}^{\frac{m}{q_{\alpha}}}(q_{\alpha}t),
\end{split}
\end{equation}
where \(\sech t= \frac{2}{e^{-t}+e^t}\) is the hyperbolic secant and
\[
q_{\alpha}=\frac{\alpha+2}{2}.
\]
Note that \(W_{p_{\alpha_k}}\) corresponds exactly to the radial solution \(U_{\alpha_k}\) in \eqref{eq:Ualpha}. The linearized operator at \(W_{p_{\alpha}}(t)\) reads
\begin{equation}
    \label{lopt}
    L_{p_{\alpha}} = \mathcal P - p_\alpha W_{p_{\alpha}}^{p_\alpha-1}
    = \mathcal P - (m+q_{\alpha})(m+2q_{\alpha})\sech^2(q_{\alpha}t).
\end{equation}

Let \(H^{1}(\mathcal C)\) be the usual Sobolev space on the cylinder. Since \(m>0\), we shall use the equivalent norm
\[
\|f\|_{H^1(\mathcal C)}^2
:=
\int_{\mathcal C}
\left(
|\partial_t f|^2
+
|\nabla_{\mathbb S^{N-1}}f|^2
+
m^2 f^2
\right)\,dt\,d\theta.
\]
Under the Emden--Fowler transform \eqref{clvar}--\eqref{clfun}, the norm \(\|f\|_{H^1(\mathcal C)}\) is equivalent to the \(D^{1,2}(\mathbb R^N)\)-norm of \(u\), as used in \cite{GGN}. 

The kernel space
\[
\ker L_{p_{\alpha}}
=
\left\{ f\in H^1(\mathcal C) : L_{p_{\alpha}} f=0 \text{ in the distribution sense} \right\}
\]
was characterized in Theorem 1.3 of \cite{GGN} in the original coordinates. In cylindrical coordinates, it can be reformulated as follows.

\begin{theorem}\label{thm:kernel}
Let \(W_{p_\alpha}\) and \(L_{p_{\alpha}}\) be as in \eqref{eq:Wp-general} and \eqref{lopt}, respectively. Then:
\begin{itemize}
    \item If \(q_{\alpha}\notin\mathbb N\), that is, if \(\alpha>0\) is not an even integer, then
\[
\ker L_{p_{\alpha}}=\operatorname{span}\{W_{p_{\alpha}}'(t)\}.
\]
\item If \(q_{\alpha}=k\in\mathbb N\), that is, if \(\alpha=2(k-1)\), then
\[
\ker L_{p_{\alpha}}
=
\operatorname{span}\{W_{p_{\alpha}}'(t)\}
\oplus
\left\{
\varphi_k(t)Y(\theta) : Y\in \mathcal Y_k
\right\},
\]
where \(\mathcal Y_k\) is the space of spherical harmonics of degree \(k\) on \(\mathbb S^{N-1}\), and
\begin{equation}
    \label{phik}
    \varphi_k(t)=\sech^{\frac{m+k}{k}}(kt),\qquad t\in\mathbb R.
\end{equation}
\end{itemize}
\end{theorem}

We shall give a short proof of the above theorem using the classical spectral theory of the P\"oschl--Teller operator in quantum mechanics. For \(\nu>0\), let
\begin{equation} \label{eq:PT-operator}
H_\nu
=
-\frac{d^2}{d\tau^2}-\nu(\nu+1)\sech^2\tau, \qquad \tau\in\mathbb R,
\end{equation}
which is a self-adjoint operator on \(L^2(\mathbb R)\) with domain \(H^2(\mathbb R)\).

For \(\rho>0\), define
\[
A_\rho=\frac{d}{d\tau}+\rho\tanh\tau,
\qquad
A_\rho^*=-\frac{d}{d\tau}+\rho\tanh\tau.
\]

\begin{lemma}\label{lem:PT-spectrum}
The negative eigenvalues of \(H_\nu\) are exactly
\begin{equation*}
E_j=-(\nu-j)^2,
\qquad
j\in \{0\}\cup\mathbb N \text{ and } \nu-j>0.
\end{equation*}
Moreover, the eigenspace corresponding to \(E_j\) is one-dimensional and is generated by
\[
\psi_j^{(\nu)}
=
A_\nu^*A_{\nu-1}^*\cdots A_{\nu-j+1}^*
\bigl(\sech^{\nu-j}\tau\bigr),
\]
with the convention that the empty product is the identity. In particular, the eigenspace for \(E_0=-\nu^2\) is generated by \(\sech^\nu\tau\).
\end{lemma}

\begin{proof}
For the reader's convenience, we recall a proof following \cite[Section 9.1]{Teschl}.
A direct computation gives
\begin{equation*}
A_\rho^*A_\rho
=
-\frac{d^2}{d\tau^2}+\rho^2-\rho(\rho+1)\sech^2\tau
=
H_\rho+\rho^2
\end{equation*}
and
\begin{equation*}
A_\rho A_\rho^*
=
-\frac{d^2}{d\tau^2}+\rho^2-\rho(\rho-1)\sech^2\tau
=
H_{\rho-1}+\rho^2.
\end{equation*}
Thus
\begin{equation}\label{eq:intertwining}
H_\rho A_\rho^*=A_\rho^*H_{\rho-1}
\quad\text{and}\quad
A_\rho H_\rho=H_{\rho-1}A_\rho.
\end{equation}
Since \(A_\rho \sech^\rho\tau=0\), we have
\[
H_\rho \sech^\rho\tau
=
(A_\rho^*A_\rho-\rho^2)\sech^\rho\tau
=
-\rho^2 \sech^\rho\tau.
\]
For \(j\in \{0\}\cup\mathbb N\) and \(\nu-j>0\), define
\begin{equation}\label{eq:psi-j-def}
\psi_j^{(\nu)}(\tau)
:=
A_\nu^*A_{\nu-1}^*\cdots A_{\nu-j+1}^*
\bigl(\sech^{\nu-j}\tau\bigr),
\end{equation}
with the convention that the empty product is the identity. Obviously, \(\psi_j^{(\nu)}\in L^2(\mathbb R)\). Since \(\ker A_\rho^*=\operatorname{span}\{\cosh^\rho\tau\}\), it is easy to see that \(\psi_j^{(\nu)}\not\equiv 0\).
Applying the first identity in \eqref{eq:intertwining} successively yields
\[
H_\nu\psi_j^{(\nu)}=-(\nu-j)^2\psi_j^{(\nu)}.
\]
Thus \(\psi_j^{(\nu)}\) are eigenfunctions associated to \(E_j=-(\nu-j)^2\).

It remains to prove that there are no other negative eigenvalues. Suppose that \(E<0\) is an eigenvalue and \(\phi\in L^2(\mathbb R)\) is an associated eigenfunction. Define the lowering sequence by
\[
\phi_0=\phi,\qquad
\phi_{j+1}=A_{\nu-j}\phi_j \quad\text{for } j\in\mathbb N,
\]
as long as the right-hand side is nonzero. By the second identity in \eqref{eq:intertwining}, whenever \(\phi_{j+1}\not\equiv 0\), we have
\[
H_{\nu-j-1}\phi_{j+1}=E\phi_{j+1}.
\]
Thus the same negative eigenvalue is transferred from \(H_\nu\) to operators with lower parameters.

If the procedure stops after \(j\) steps, namely \(A_{\nu-j}\phi_j=0\), then using
\[
A_{\nu-j}^*A_{\nu-j}=H_{\nu-j}+(\nu-j)^2
\]
we get
\[
H_{\nu-j}\phi_j=-(\nu-j)^2\phi_j.
\]
Since the eigenvalue has been preserved along the lowering procedure, it follows that \(E=-(\nu-j)^2\).

Otherwise, the procedure reaches a parameter \(\nu-j\in(0,1]\) and still has a nonzero negative eigenfunction of \(H_{\nu-j}\) which is not annihilated by \(A_{\nu-j}\). In the latter case, \(A_{\nu-j}\phi_j\) would be a nonzero eigenfunction of \(H_{\rho-1}\), where \(\rho=\nu-j\), with the same negative eigenvalue. But since \(0<\rho\le 1\),
\[
H_{\rho-1}
=
-\frac{d^2}{d\tau^2}
+\rho(1-\rho)\sech^2\tau.
\]
For \(v\in C_c^\infty(\mathbb R)\), integration by parts gives
\[
\begin{aligned}
\langle H_{\rho-1}v,v\rangle
&=
\int_{\mathbb R}
\left(
-v''+\rho(1-\rho)\sech^2\tau\,v
\right)v\,d\tau \\
&=
\int_{\mathbb R}|v'|^2\,d\tau
+
\rho(1-\rho)
\int_{\mathbb R}\sech^2\tau\, |v|^2\,d\tau
\ge 0.
\end{aligned}
\]
By density, the same identity defines the quadratic form for all \(v\in H^1(\mathbb R)\); hence \(H_{\rho-1}\) has nonnegative quadratic form. It follows that \(H_{\rho-1}\) has no negative \(L^2\)-eigenvalue. We obtain a contradiction, and hence the lowering procedure must stop.

Finally, we prove that each \(L^2\)-eigenspace is one-dimensional. Indeed, let \(u,v\in L^2(\mathbb R)\) solve
\begin{equation}
    \label{Heq}
    H_\nu u=Eu,\qquad H_\nu v=Ev,\qquad E<0.
\end{equation}
Their Wronskian is
\[
\mathcal W(u,v):=uv'-u'v.
\]
By \eqref{Heq}, we have
\[
\mathcal W'(u,v)=0,
\]
so it is constant. Since \(E<0\) and the potential
\(-\nu(\nu+1)\sech^2\tau\) in \eqref{eq:PT-operator} tends to \(0\) as \(|\tau|\to\infty\), every
\(L^2\)-solution decays exponentially at both ends. Therefore
\[
\mathcal W(u,v)(\tau)\to 0
\quad\text{as }|\tau|\to\infty.
\]
Thus \(\mathcal W(u,v)\equiv 0\), and \(u\) and \(v\) are linearly dependent. Hence each negative \(L^2\)-eigenspace of \(H_\nu\) is one-dimensional.

Since we have already constructed a nonzero eigenfunction
\(\psi_j^{(\nu)}\) for \(E_j=-(\nu-j)^2\), this eigenfunction spans the
corresponding eigenspace. In particular, for \(j=0\), the eigenspace of
\(E_0=-\nu^2\) is generated by \(\sech^\nu\tau\).
The lemma is proved.
\end{proof}

\begin{proof}[Proof of Theorem~\ref{thm:kernel}]
Suppose \(v\in\ker L_{p_{\alpha}}\). By elliptic regularity, \(v\) is smooth. By Fourier expansion,
\[
v(t,\theta)=\sum_{l\ge0} f_l(t)Y_l(\theta),
\]
where \(f_l\in L^2(\mathbb R)\) and \(Y_l\) denotes the \(l\)-th spherical harmonic, satisfying
\[
-\Delta_{\mathbb S^{N-1}}Y_l=l(l+2m)Y_l.
\]
Each coefficient \(f_l\) satisfies
\begin{equation}\label{eq:modeq}
-f_l''+\bigl[(m+l)^2-(m+q_{\alpha})(m+2q_{\alpha})\sech^2(q_{\alpha}t)\bigr]f_l=0.
\end{equation}
Setting \(\tau=q_{\alpha}t\) and \(\nu=\frac{m+q_{\alpha}}{q_{\alpha}}\), we immediately find that \eqref{eq:modeq} becomes
\begin{equation}
    \label{Hnueq}
    H_{\nu}f_l
=
-\left(\frac{m+l}{q_{\alpha}}\right)^2f_l.
\end{equation}
Since \(v\in H^1(\mathcal C)\), each nonzero coefficient \(f_l\in L^2(\mathbb R)\). Hence, by Lemma \ref{lem:PT-spectrum},
\[
\left(\frac{m+l}{q_{\alpha}}\right)^2=(\nu-j)^2
\quad\text{for some } j\in\{0\}\cup\mathbb N \text{ with } \nu-j>0.
\]
Thus \(l=q_{\alpha}(1-j)\ge 0\), which implies that either \(l=0\) or \(l=q_{\alpha}\).

If \(l=0\), this is the radial cylindrical mode. Since \(W_{p_\alpha}'\) is one such mode, the one-dimensionality gives
\[
f_0(t)\in \operatorname{span}\{W_{p_\alpha}'(t)\}.
\]
The case \(l=q_{\alpha}\) occurs only when \(\alpha=2(k-1)\) for some \(k\in\mathbb N\). In that case, \(j=0\) and \(q_\alpha=k\). Recalling \eqref{eq:modeq}, we have
\[
-f_k''+\left[(m+k)^2-(m+k)(m+2k)\sech^2(kt)\right]f_k=0.
\]
Setting \(\tau=kt\) and \(\nu=\frac{m+k}{k}\), the equation \eqref{Hnueq} and Lemma \ref{lem:PT-spectrum} imply that the eigenspace corresponding to the eigenvalue \(-\nu^2\) is one-dimensional and is generated by \(\sech^\nu\tau\).
Therefore, up to a constant factor,
\[
f_k(t)=\sech^\nu(kt)
=
\sech^{\frac{m+k}{k}}(kt),
\]
and the angular factor is an arbitrary spherical harmonic of degree \(k\).
The theorem is proved.
\end{proof} 

\section{Direct local bifurcation}
\label{sec:local}

Let $\alpha_k = 2(k-1)$ with $k > \frac{N-2}{2}$, so that
\[
p_{\alpha_k} = 1 + \frac{2k}{m} > 3.
\]
Choose $\gamma$ and $\sigma$ such that
\begin{equation}\label{eq:gamma-choice}
0 < \gamma < \min\{1, p_0-1, p_{\alpha_k}-3\},
\end{equation}
and
\begin{equation}\label{eq:sigma-general}
\frac{m}{p_0} < \sigma < m.
\end{equation}
For $j \ge 0$, define the weighted H\"older space
\[
C^{j,\gamma}_\sigma(\mathcal C)
:=
\left\{
v \in C^{j,\gamma}_{\mathrm{loc}}(\mathcal C) :
\|v\|_{C^{j,\gamma}_\sigma} < \infty
\right\},
\]
equipped with the norm
\begin{equation}\label{vnorm}
\|v\|_{C^{j,\gamma}_\sigma}
:=
\sup_{\tau \in \mathbb R}
e^{\sigma|\tau|}
\|v\|_{C^{j,\gamma}((\tau-1,\tau+1) \times \mathbb S^{N-1})},
\end{equation}
where $C^{j,\gamma}((\tau-1,\tau+1) \times \mathbb S^{N-1})$ denotes the standard H\"older space.

A straightforward calculation yields the following Nemytskii-type result.

\begin{lemma}\label{lem:nemytskii-local}
Let $[\underline p, \overline p] \subset (3,\infty)$ and assume that $\gamma$ in \eqref{eq:gamma-choice} satisfies $\gamma < \underline p - 3$. Then the map
\[
(p,v) \longmapsto (v_+)^p
\]
is of class $C^3$ from $[\underline p, \overline p] \times C^{2,\gamma}_\sigma(\mathcal C)$ to $C^{0,\gamma}_\sigma(\mathcal C)$, where $v_+ := \max\{v,0\}$. Moreover, its derivatives are given by
\begin{align*}
D_v(v_+)^p[h] &= p(v_+)^{p-1}h, \\
D_v^2(v_+)^p[h_1,h_2] &= p(p-1)(v_+)^{p-2}h_1h_2, \\
D_p(v_+)^p &= (v_+)^p \log(v_+),
\end{align*}
with the convention that $s^a(\log s)^b = 0$ at $s=0$ whenever $a>0$.
\end{lemma}

For sufficiently small $\delta > 0$ such that $|p_{\alpha} - p_{\alpha_k}| < \delta$, we may assume that the interval $J := [p_{\alpha_k}-\delta, p_{\alpha_k}+\delta]$ is compactly contained in $(3,\infty)$, that $\gamma < p_{\alpha_k}-\delta-3$, and that $(p_{\alpha_k}-\delta)\sigma > m$. Define the nonlinearity
\[
g_{p_{\alpha}}(s) := (s_+)^{p_{\alpha}}.
\]
Let $W_{p_{\alpha}}$ be the homoclinic solution given by \eqref{eq:Wp-general}, and define the operator
\begin{equation}\label{eq:F-general-def}
\begin{split}
\mathcal F(p_{\alpha},\eta)
&:= \mathcal P(W_{p_{\alpha}}+\eta) - g_{p_{\alpha}}(W_{p_{\alpha}}+\eta) \\
&= \mathcal P\eta - \bigl[g_{p_{\alpha}}(W_{p_{\alpha}}+\eta) - W_{p_{\alpha}}^{p_{\alpha}}\bigr]
\end{split}
\end{equation}
for $(p_{\alpha},\eta) \in J \times C^{2,\gamma}_\sigma(\mathcal C)$.
Since $W_{p_\alpha}$ is smooth with respect to $p_\alpha$ and satisfies $D_{p_{\alpha}}^j W_{p_{\alpha}}(t) = \mathrm O(e^{-m|t|})$ for every $j \ge 0$, it follows from Lemma~\ref{lem:nemytskii-local} and our choice of $\sigma$ and $\gamma$ that $\mathcal F \colon J \times C^{2,\gamma}_\sigma(\mathcal C) \to C^{0,\gamma}_\sigma(\mathcal C)$ is a $C^3$ map.

Observe that
\begin{equation}\label{eq:trivial}
\mathcal F(p_\alpha,0) = 0,
\end{equation}
and
\begin{equation}\label{eq:Lkdef}
D_\eta \mathcal F(p_{\alpha_k},0) = L_{p_{\alpha_k}}
= \mathcal P - p_{\alpha_k} W_{p_{\alpha_k}}^{p_{\alpha_k}-1}.
\end{equation}
Let
\[
\mathcal I
:=
\left\{
\lambda \in \mathbb R :
\mathcal P(e^{\lambda t}\phi) = 0
\text{ for some } 0 \not\equiv \phi \in C^\infty(\mathbb S^{N-1})
\right\}
\]
denote the set of indicial roots of $\mathcal P$. A direct computation using spherical harmonics shows that
\[
\mathcal I = \{\pm(m+l) : l = 0,1,2,\dots\}.
\]
Since the coefficients of $L_{p_{\alpha_k}}$ converge to those of $\mathcal P$ as $|t| \to \infty$, we may also refer to $\mathcal I$ as the set of indicial roots of $L_{p_{\alpha_k}}$.
By our choice of $\sigma \notin \mathcal I$,
\begin{equation}\label{eq:P-isom}
\begin{split}
\mathcal P :&\; C^{2,\gamma}_\sigma(\mathcal C) \to C^{0,\gamma}_\sigma(\mathcal C) \quad \text{is surjective}, \\
&\|v\|_{C^{2,\gamma}_\sigma(\mathcal C)} \le C \|\mathcal P v\|_{C^{0,\gamma}_\sigma(\mathcal C)}
\quad \text{for all } v \in C^{2,\gamma}_\sigma(\mathcal C),
\end{split}
\end{equation}
where $C>0$ depends only on $N$, $\sigma$, and $\gamma$.
Indeed, this follows from Proposition 6.2.1 and Lemma 12.1.1 of Pacard \cite{P}.
More importantly, 
\[
L_{p_{\alpha_k}} : C^{2,\gamma}_\sigma(\mathcal C) \to C^{0,\gamma}_\sigma(\mathcal C)
\quad \text{is Fredholm}.
\]
It follows from the weighted Sobolev Fredholm theorem of
Lockhart--McOwen \cite[Theorem 8.1]{LM} that
\(L_{p_{\alpha_k}}\) is Fredholm in the corresponding weighted Sobolev spaces,
since \(\sigma\) avoids the indicial roots.  By standard elliptic regularity
and the weighted Schauder estimate \eqref{eq:P-isom}, the same Fredholm
property holds in the weighted H\"older spaces. Indeed, Sobolev solutions with
H\"older right-hand side are upgraded to \(C^{2,\gamma}_\sigma\), and the
adjoint kernel defining the solvability conditions is smooth and has the
corresponding exponential decay.

To apply the classical Crandall--Rabinowitz theorem, we introduce the following closed subspaces with restricted symmetries:
\begin{align*}
X := \bigl\{ v \in C^{2,\gamma}_\sigma(\mathcal C) : \;
&v(t,\theta) = v(-t,\theta), \\
&v(t,A\theta)=v(t,\theta)\ \text{for all }A\in O(N-1)\times O(1) \bigr\},
\end{align*}
and
\begin{align*}
Y := \bigl\{ v \in C^{0,\gamma}_\sigma(\mathcal C) : \;
&v(t,\theta) = v(-t,\theta), \\
&v(t,A\theta)=v(t,\theta)\ \text{for all }A\in O(N-1)\times O(1) \bigr\},
\end{align*}
where \(O(N-1)\) acts on the first \(N-1\) variables and \(O(1)\) acts on the last variable \(\theta_N\). These are closed subspaces of $C^{2,\gamma}_\sigma(\mathcal C)$ and $C^{0,\gamma}_\sigma(\mathcal C)$, respectively. It is evident that $\mathcal F$ maps $(p_{\alpha_k}-\delta, p_{\alpha_k}+\delta) \times X$ into $Y$. Henceforth, we restrict our attention to this setting.

Since $L_{p_{\alpha_k}}$ commutes with the reflection $t \mapsto -t$ and with the $O(N-1)$-action on the sphere, the symmetry spaces $X$ and $Y$ are invariant under $L_{p_{\alpha_k}}$. Since these symmetry spaces are closed subspaces obtained by bounded projections, the restriction
\begin{equation}\label{eq:Fredholm-general}
L_{p_{\alpha_k}} : X \to Y \quad \text{is also Fredholm}.
\end{equation}

\begin{lemma}\label{Find}
The Fredholm index of the operator $L_{p_{\alpha_k}} \colon X \to Y$ is zero.
\end{lemma}

\begin{proof}
For $s \in [0,1]$, define the continuous family of operators
\[
L_{p_{\alpha_k},s} := \mathcal P - s \, p_{\alpha_k} W_{p_{\alpha_k}}^{p_{\alpha_k}-1}.
\]
As $|t| \to \infty$, the coefficients of $L_{p_{\alpha_k},s}$ converge uniformly to those of $\mathcal P$; hence they share the same indicial roots. Since $\sigma \notin \mathcal I$, it follows from Theorem~8.1 of Lockhart--McOwen \cite{LM} that
\[
L_{p_{\alpha_k},s} \colon X \to Y, \qquad 0 \le s \le 1,
\]
is a Fredholm operator.

The Fredholm index is invariant under continuous deformations within the family of Fredholm operators; see, for example, Conway \cite[Chapter~XI, Sections~2--3]{Conway} or Kato \cite[Chapter~IV]{Kato}. Therefore,
\[
\operatorname{ind}(L_{p_{\alpha_k}})
=
\operatorname{ind}(L_{p_{\alpha_k},1})
=
\operatorname{ind}(L_{p_{\alpha_k},0})
=
\operatorname{ind}(\mathcal P).
\]
By \eqref{eq:P-isom}, we have $\operatorname{ind}(\mathcal P) = 0$, which completes the proof.
\end{proof}

\begin{lemma}\label{KerLk}
The kernel of $L_{p_{\alpha_k}} \colon X \to Y$ is one-dimensional.
\end{lemma}

\begin{proof}
According to Theorem~\ref{thm:kernel}, the kernel is given by
\[
\ker(L_{p_{\alpha_k}} \colon X \to Y)
=
X \cap \left(
\operatorname{span}\{W_{p_{\alpha}}'(t)\}
\oplus
\left\{ \varphi_k(t)Y(\theta) : Y \in \mathcal Y_k \right\}
\right).
\]
We analyze the two components of this direct sum separately. First, $W_{p_{\alpha}}'(t) \notin X$ because it is not even in $t$. Second, the $O(N-1)$-invariant spherical harmonics of degree $k$ are precisely the zonal harmonics, which form a one-dimensional space; see Stein--Weiss \cite[ p.~149]{SW71}. Let $Y_k(\theta_N)$ be a nontrivial zonal harmonic of degree $k$. Then
\begin{equation}\label{eq:real-ker}
\Phi_k(t,\theta_N)=\varphi_k(t)Y_k(\theta_N)
=
\sech^{\frac{m+k}{k}}(kt) Y_k(\theta_N).
\end{equation}
In fact, we may choose
\begin{equation}
Y_k(\theta_N)=C_k^{(m)}(\theta_N),
\end{equation}
where $C_j^{(m)}$ denotes the Gegenbauer polynomial of degree $j$ with parameter $m=(N-2)/2$. Since \(k\) is even, this zonal harmonic is also invariant under \(\theta_N\mapsto-\theta_N\).

Consequently, the intersection with $X$ is one-dimensional, and
\[
\ker(L_{p_{\alpha_k}} \colon X \to Y) = \operatorname{span}\{\Phi_k(t,\theta_N)\}.
\]
The lemma is proved.
\end{proof}

\begin{lemma}\label{lem:cubicmoment}
Let $N\ge 3$ and let $k\ge 2$ be even. For $Y_k(\theta_N)=C_k^{(m)}(\theta_N)$,
\begin{equation}\label{eq:angular-positive}
\int_{\mathbb S^{N-1}} Y_k^3 \, d\theta > 0.
\end{equation}
\end{lemma}

\begin{proof}
For any function $F$ depending only on $x=\theta_N$,
\begin{equation}\label{eq:sphintegration}
\int_{\mathbb S^{N-1}} F(\theta_N)\,d\theta
=
|\mathbb S^{N-2}| \int_{-1}^1 F(x)(1-x^2)^{m-\frac12}\,dx.
\end{equation}
The Gegenbauer polynomials are orthogonal with respect to this weight. By the Gegenbauer linearization formula and Gasper's positivity theorem for Jacobi polynomial linearization coefficients \cite[Theorem 1]{Gasper}; see also \cite[Section 6.8]{AAR},
\begin{equation}\label{eq:linearization}
C_k^{(m)}(x)C_k^{(m)}(x)=\sum_{j=0}^k b_j C_{2k-2j}^{(m)}(x),\qquad x\in[-1,1],
\end{equation}
with $b_j>0$ for all indices allowed by parity. Since $k$ is even, the term $C_k^{(m)}$ occurs in \eqref{eq:linearization}, namely for $j=k/2$, with strictly positive coefficient. Multiplying by $C_k^{(m)}$ and using orthogonality leaves only this term. The lemma is proved.
\end{proof}

\begin{lemma} \label{lem:codim}
\[
\operatorname{Range} L_{p_{\alpha_k}} := L_{p_{\alpha_k}}(X)
=
\left\{ f \in Y : \int_{\mathcal C} f \Phi_k \, dt \, d\theta = 0 \right\}.
\]
\end{lemma}

\begin{proof}
For $f \in Y$, we have $f = O(e^{-\sigma |t|})$ and $\Phi_k = O(e^{-(m+k)|t|})$; hence the integral is well-defined.
Moreover, the linear functional
\[
\ell(f) := \int_{\mathcal C} f \Phi_k \, dt \, d\theta
\]
is nontrivial on $Y$, since $\ell(\Phi_k)>0$. Thus $\ker \ell$ is a closed subspace of $Y$ of codimension one.

We first show that
\[
L_{p_{\alpha_k}}(X) \subset \ker \ell.
\]
Let $u \in X$. Using the self-adjointness of $L_{p_{\alpha_k}}$ together with $L_{p_{\alpha_k}}\Phi_k = 0$, integration by parts on $(-R,R) \times \mathbb S^{N-1}$ gives
\begin{equation}
\label{locint}
\int_{(-R,R)\times \mathbb S^{N-1}}
(L_{p_{\alpha_k}}u)\Phi_k \, dt \, d\theta
=
\int_{(-R,R)\times \mathbb S^{N-1}}
u (L_{p_{\alpha_k}}\Phi_k) \, dt \, d\theta
+ B_R
=
B_R,
\end{equation}
where
\[
B_R
=
\int_{\mathbb S^{N-1}}
\left[
-u_t \Phi_k + u (\Phi_k)_t
\right]_{t=-R}^{t=R}
d\theta.
\]
Moreover, we have
\[
|u| + |u_t| = O(e^{-\sigma |t|}),
\qquad
|\Phi_k| + |(\Phi_k)_t| = O(e^{-(m+k)|t|}).
\]
Therefore
\[
|B_R| \le C e^{-(m+k+\sigma)R} \to 0
\qquad \text{as } R \to \infty.
\]
Letting $R \to \infty$ in \eqref{locint}, we obtain
\[
\int_{\mathcal C} (L_{p_{\alpha_k}}u)\Phi_k \, dt \, d\theta = 0,
\]
hence
\[
L_{p_{\alpha_k}}(X) \subset \ker \ell.
\]

On the other hand, it follows from Lemmas \ref{Find} and \ref{KerLk} that
\[
\operatorname{codim} \operatorname{Range} L_{p_{\alpha_k}}
=
\dim \ker L_{p_{\alpha_k}}
-
\operatorname{ind} L_{p_{\alpha_k}}
=
1.
\]
Therefore $\operatorname{Range} L_{p_{\alpha_k}}$ is a codimension-one subspace of $Y$. Since it is contained in the codimension-one subspace $\ker \ell$, the two spaces must be equal.
\end{proof}

\begin{proposition}\label{prop:localbif}
There exist $\varepsilon > 0$, a $C^2$ curve $s \mapsto (p(s),\eta(s)) \in \mathbb R \times X$ for $s \in (-\varepsilon,\varepsilon)$, and a closed complement $Z$ of $\operatorname{span}\{\Phi_k\}$ in $X$ such that
\begin{equation}\label{eq:CRcurve}
p(0) = p_{\alpha_k}, \quad \eta(0) = 0, \quad \eta(s) = s\Phi_k + s\psi(s), \quad \psi(s) \in Z, \quad \psi(0) = 0,
\end{equation}
and
\begin{equation}\label{eq:Fzero-branch}
\mathcal F(p(s),\eta(s)) = 0.
\end{equation}
The local nontrivial solutions of \eqref{eq:Fzero-branch} near $(p_{\alpha_k},0)$ are exhausted by this curve together with the trivial curve $\eta = 0$.
\end{proposition}

\begin{proof}
To apply the Crandall--Rabinowitz theorem \cite{CR}, it remains to verify the transversality condition. Recall
\begin{equation*}
L_{p_{\alpha}} = D_\eta \mathcal F(p_{\alpha},0) = \mathcal P - p_{\alpha} W_{p_{\alpha}}^{p_{\alpha}-1}.
\end{equation*}
From \eqref{eq:Wp-general}, we have
\begin{equation}
\label{palphafm}
p_{\alpha} W_{p_{\alpha}}^{p_{\alpha}-1}
=
\frac{m^2 p_{\alpha}(p_{\alpha}+1)}{2}
\sech^2\left(\frac{m(p_{\alpha}-1)}{2} t\right).
\end{equation}
We first compute the sign of the transversality quantity
\[
D_k
:=
\left\langle
\partial_{p_{\alpha}} L_{p_{\alpha}}\big|_{p_{\alpha}=p_{\alpha_k}}
\Phi_k,
\Phi_k
\right\rangle_{L^2(\mathcal C)}.
\]
We claim that
\begin{equation}\label{eq:Dknegative-short}
D_k < 0.
\end{equation}
To prove this, set
\[
A(p_{\alpha}) = \frac{m^2 p_{\alpha}(p_{\alpha}+1)}{2}, \qquad b(p_{\alpha}) = \frac{m(p_{\alpha}-1)}{2}.
\]
Then $b(p_{\alpha_k}) = k$ and $A(p_{\alpha_k}) = (m+k)(m+2k)$. Hence, by \eqref{palphafm},
\[
\partial_{p_{\alpha}} \big(p_{\alpha} W_{p_{\alpha}}^{p_{\alpha}-1}\big)\big|_{p_{\alpha}=p_{\alpha_k}}
=
A'(p_{\alpha_k})\sech^2(kt) - m A(p_{\alpha_k}) t \sech^2(kt)\tanh(kt).
\]
Set $Q = 2(m+2k)/k$. Then $\varphi_k^2 \sech^2(kt) = \sech^Q(kt)$ and
\[
\frac{d}{dt} \sech^Q(kt) = -Qk \sech^Q(kt)\tanh(kt),
\]
so integration by parts yields
\[
\int_{\mathbb R} t \sech^Q(kt)\tanh(kt)\,dt
=
\frac{1}{Qk} \int_{\mathbb R} \sech^Q(kt)\,dt > 0.
\]
Since \(p_{\alpha_k}=1+\frac{2k}{m}\), we have
\[
A'(p_{\alpha_k})
=\frac{m^2}{2}\left(3+\frac{4k}{m}\right)
=\frac{m(3m+4k)}{2}.
\]
On the other hand, by the definition of \(Q\),
\[
\frac{mA(p_{\alpha_k})}{Qk}
=
\frac{m(m+k)(m+2k)}{2(m+2k)}
=
\frac{m(m+k)}{2},
\]
and
\[
A'(p_{\alpha_k})-\frac{mA(p_{\alpha_k})}{Qk}
=
\frac{m(2m+3k)}{2}.
\]
Therefore,
\begin{align*}
\int_{\mathbb R}
\partial_{p_{\alpha}} \big(p_{\alpha} W_{p_{\alpha}}^{p_{\alpha}-1}\big)\big|_{p_{\alpha}=p_{\alpha_k}}
\varphi_k^2 \, dt
&=
\left(A'(p_{\alpha_k}) - \frac{m A(p_{\alpha_k})}{Qk}\right)
\int_{\mathbb R} \sech^Q(kt)\,dt \\
&=
\frac{m(2m+3k)}{2}
\int_{\mathbb R} \sech^Q(kt)\,dt > 0.
\end{align*}
Combining $\partial_{p_{\alpha}} L_{p_{\alpha}} = -\partial_{p_{\alpha}} \big(p_{\alpha} W_{p_{\alpha}}^{p_{\alpha}-1}\big)$ with \eqref{eq:real-ker} establishes \eqref{eq:Dknegative-short}. Finally, by Lemma~\ref{lem:codim}, we obtain
\begin{equation}\label{eq:transversality}
\partial_{p_{\alpha}} L_{p_{\alpha}}\big|_{p=p_{\alpha_k}} \Phi_k \notin \operatorname{Range} L_{p_{\alpha_k}}.
\end{equation}
Thus the Crandall--Rabinowitz theorem \cite[Theorem 1.7]{CR} applies, completing the proof.
\end{proof}

\begin{proposition}\label{prop:non-vert}
Under the assumptions of Proposition~\ref{prop:localbif} and with $k$ even, we have $p'(0) < 0$. Consequently,
\[
\alpha(s) = m(p(s)-1) - 2 = \alpha_k + m p'(0) s + \mathrm O(s^2),
\]
with $p'(0) < 0$. In particular, the bifurcation branch in Proposition~\ref{prop:localbif} is not vertical with respect to the parameter $\alpha$.
\end{proposition}

\begin{proof}
Since $\mathcal F$ is of class $C^3$, the Crandall--Rabinowitz bifurcation curve is $C^2$. We may therefore set $\nu := p'(0)$ and $\Psi := \frac{1}{2}\eta''(0)$, and write the Taylor expansions
\begin{equation}
    \label{petae}
    \begin{aligned}
        p(s) &= p_{\alpha_k} + \nu s + \mathrm O(s^2), \\
\eta(s) &= s\Phi_k + s^2\Psi + o(s^2)
    \end{aligned}
\end{equation}
in $\mathbb R \times X$. 

By \eqref{lopt} and \eqref{eq:F-general-def}, the Taylor expansion of \(g_p\) at \(W_p\) gives
\[
\mathcal F(p,\eta)
=
L_p\eta
-
\left[
(W_p+\eta)^p-W_p^p-pW_p^{p-1}\eta
\right].
\]
By \eqref{petae}, we have
\[
L_{p(s)}\eta(s)
=
sL_{p_{\alpha_k}}\Phi_k
+
s^2\left(
L_{p_{\alpha_k}}\Psi
+
\nu\,\partial_{p_\alpha}L_{p_\alpha}\big|_{p_\alpha=p_{\alpha_k}}\Phi_k
\right)
+
o(s^2).
\]
Here the first-order term vanishes since \(L_{p_{\alpha_k}}\Phi_k=0\).
On the other hand, Taylor expansion of the nonlinear remainder gives
\[
(W_{p(s)}+\eta(s))^{p(s)}
-W_{p(s)}^{p(s)}
-p(s)W_{p(s)}^{p(s)-1}\eta(s)
=
\frac12 p_{\alpha_k}(p_{\alpha_k}-1)W_{p_{\alpha_k}}^{p_{\alpha_k}-2}s^2\Phi_k^2
+
o(s^2).
\]
Indeed, in the quadratic remainder we may freeze the coefficients at
\(p=p_{\alpha_k}\).  Since
\[
p(s)-p_{\alpha_k}=O(s),\qquad W_{p(s)}-W_{p_{\alpha_k}}=O(s),
\qquad \eta(s)^2=s^2\Phi_k^2+o(s^2),
\]
the changes of \(p(s)\) and \(W_{p(s)}\) in the coefficient of \(\eta(s)^2\)
produce only \(O(s^3)\) terms.  Hence
\[
\frac12 p(s)(p(s)-1)W_{p(s)}^{p(s)-2}\eta(s)^2
=
\frac12 p_{\alpha_k}(p_{\alpha_k}-1)W_{p_{\alpha_k}}^{p_{\alpha_k}-2}s^2\Phi_k^2
+o(s^2).
\]
 Therefore, expanding the equation
\(\mathcal F(p(s),\eta(s))=0\) up to order \(s^2\) yields
\[
L_{p_{\alpha_k}}\Psi
+
\nu\,\partial_{p_\alpha}L_{p_\alpha}\big|_{p_\alpha=p_{\alpha_k}}\Phi_k
-
\frac12 p_{\alpha_k}(p_{\alpha_k}-1)W_{p_{\alpha_k}}^{p_{\alpha_k}-2}\Phi_k^2
=
0.
\]

Taking the $L^2(\mathcal C)$ inner product with $\Phi_k$, and using the self-adjointness of $L_{p_{\alpha_k}}$ together with $L_{p_{\alpha_k}}\Phi_k = 0$, we obtain
\begin{equation}\label{eq:slopeeq}
\nu D_k
=
\frac{1}{2} p_{\alpha_k}(p_{\alpha_k}-1)
\int_{\mathcal C} W_{p_{\alpha_k}}^{p_{\alpha_k}-2}\Phi_k^3 \, dt \, d\theta,
\end{equation}
where $D_k < 0$ is defined in \eqref{eq:Dknegative-short}. The integral on the right-hand side separates as
\begin{equation}\label{eq:separated}
\int_{\mathcal C} W_{p_{\alpha_k}}^{p_{\alpha_k}-2}\Phi_k^3 \, dt \, d\theta
=
\left(
\int_{\mathbb R} W_{p_{\alpha_k}}^{p_{\alpha_k}-2}\varphi_k^3 \, dt
\right)
\left(
\int_{\mathbb S^{N-1}} Y_k^3 \, d\theta
\right).
\end{equation}
The radial integral is strictly positive. Furthermore, by Lemma~\ref{lem:cubicmoment}, the angular factor is also positive. Combining this with \eqref{eq:slopeeq}, \eqref{eq:separated}, and the fact that $D_k < 0$, we conclude that
\begin{equation}\label{eq:nunonzero}
\nu = p'(0) < 0.
\end{equation}
This completes the proof.
\end{proof}

\begin{proof}[Proof of Theorem~\ref{thm:main}]
By Propositions~\ref{prop:localbif} and~\ref{prop:non-vert}, for every \(\alpha_k=2(k-1)\) with \(k>(N-2)/2\) an even integer, there exists a \(C^2\) local bifurcation curve \(s\longmapsto (p(s),\eta(s))\)
with 
\[ p(0)=p_{\alpha_k},\qquad \eta(0)=0,\qquad \eta(s)\not\equiv 0 \quad\text{for }s\ne0. \] 
Moreover, Proposition~\ref{prop:non-vert} gives \(p'(0)<0\). Hence, after reducing the neighborhood if necessary, \(s\mapsto p(s)\) is locally invertible. Since \(\alpha=m(p_{\alpha}-1)-2\), the same local branch can be parametrized by \(\alpha\). Therefore, for \(0<|\alpha-\alpha_k|<\varepsilon\), \(\varepsilon\) sufficiently small, there exists \(\eta_\alpha\in X\) with non-vanishing angular derivative such that \[ \mathcal F(p_\alpha,\eta_\alpha)=0. \] At \(\alpha=\alpha_k\), the branch passes through the radial solution \(U_{\alpha_k}\).

Define $w_{\alpha} := W_{p_{\alpha}} + \eta_\alpha$. This function satisfies
\[
\mathcal P w_{\alpha} = (w_{\alpha})_+^{p_{\alpha}} \ge 0 \quad \text{on } \mathcal C,
\]
and decays as
\[
|w_{\alpha}(t)| = \mathrm O(e^{-\sigma |t|}) \quad \text{as } |t| \to \infty.
\]
By the strong maximum principle, $w_\alpha > 0$ on $\mathcal C$. Hence $w_\alpha$ is a nontrivial solution of \eqref{eq:cylinder-general}. Substituting this back into the equation and applying the Fourier analysis on the
cylinder (see Section 7 of \cite{P}), we obtain the sharper decay estimate
\begin{equation}
    \label{wdecay}
    w_{\alpha}(t) = \mathrm O(e^{-m |t|}) \quad \text{as } |t| \to \infty.
\end{equation}
Finally, applying the transformation introduced in \eqref{clvar}--\eqref{clfun}, we set
\[
u_\alpha(x) := |x|^{-\frac{N-2}{2}} w_\alpha(t,\theta).
\]
Moreover, the estimate \eqref{wdecay} implies that
\[
u_\alpha(x)=|x|^{-m}w_\alpha \left(\log |x|, \frac{x}{|x|} \right) 
\]
is bounded near \(x=0\). Since \(\alpha>0\), the right-hand side
\(|x|^\alpha u_\alpha^{p_\alpha}\) is locally bounded near the origin. Thus the isolated point \(x=0\) is removable.  By standard elliptic regularity, \(u_\alpha\) has at least the regularity
\(u_\alpha\in C^{2,\beta}_{\mathrm{loc}}(\mathbb R^N)\cap
C^\infty(\mathbb R^N\setminus\{0\})\) for some
\(\beta\in(0,\min\{\alpha,1\})\).  Thus \(u_\alpha\) is a positive classical solution on all of \(\mathbb R^N\). The same estimate, now as \(|x|\to\infty\), gives \(u_\alpha(x)=O(|x|^{2-N})\). Since \(\eta_\alpha\) has non-vanishing angular derivative for \(0<|\alpha-\alpha_k|<\varepsilon\), the solution \(u_\alpha\) is non-radial.
\end{proof}

\section{Another proof of the bifurcation theorem of Gladiali--Grossi--Neves}
\label{sec:global}

In this section we employ a $C^1$ fixed-point framework, so the restriction \eqref{eq:gamma-choice} from Section~\ref{sec:local} is no longer needed. 
Let \(\alpha_k=2(k-1)\) with \(k>1\).  The symmetry group \(G\) is chosen as
follows.  If \(k\) is odd, we take \(G=O(N-1)\).  If \(k\) is even, then for
each \(h=1,\ldots,\left\lfloor N/2\right\rfloor\), we take
\(G_h=O(h)\times O(N-h)\), where \(\lfloor N/2\rfloor\)
denotes the greatest integer less than or equal to \(N/2\).  Equivalently,
\begin{equation}\label{eq:G}
G =
\begin{cases}
O(N-1), & \text{if } k \text{ is odd}, \\[2mm]
G_h, & \text{if } k \text{ is even}.
\end{cases}
\end{equation}
For $j = 0, 1, 2$, define the symmetry spaces
\[
\mathcal E_G^j := \left\{
v \in C^{j,\gamma}_\sigma(\mathcal C) :
v(t,\theta) = v(-t,\theta),\;
v(t,g\theta) = v(t,\theta) \text{ for all } g \in G
\right\},
\]
where $0 < \sigma < m$ and $\gamma \in (0,1)$. If $G = O(N-1)$, the action $g\theta$ is understood as in the definition of the space $X$ in the previous section.
By the Smoller--Wasserman invariant-harmonic calculation, as in \cite[Section 3.1 and Theorem 3.8]{GGN}, the space of degree-$k$ $G$-invariant spherical harmonics on $\mathbb S^{N-1}$ is one-dimensional; we denote it by
\[
\operatorname{span}\{\hat Y_k(\theta)\}.
\]

Since the bounded linear operators $\mathcal P$ and $L_{p_\alpha}$ commute with the reflection $t \mapsto -t$ and with the $G$-action on the sphere, the spaces $\mathcal E_G^2$ and $\mathcal E_G^0$ are invariant under them. In this section, we consider
\[
\mathcal P,\; L_{p_\alpha} : \mathcal E_G^2 \to \mathcal E_G^0.
\]
Arguing as in the previous section, one obtains that both $\mathcal P$ and $L_{p_\alpha}$ are Fredholm operators of index zero. Furthermore:

\begin{lemma}\label{lem:simple-crossing-global}
With the above assumptions, we have:
\begin{itemize}
\item[(i)] The kernel of $L_{p_{\alpha_k}} \colon \mathcal E_G^2 \to \mathcal E_G^0$ is one-dimensional and is given by
\[
\ker L_{p_{\alpha_k}} = \operatorname{span}\{\hat \Phi_k\},
\]
where $\hat \Phi_k(t,\theta) = \varphi_k(t)\hat Y_k(\theta)$ and $\varphi_k$ is as in \eqref{phik}.

\item[(ii)]
\[
L_{p_{\alpha_k}}(\mathcal E_G^2)
=
\left\{
f \in \mathcal E_G^0 :
\int_{\mathcal C} f \, \hat \Phi_k \, dt \, d\theta = 0
\right\}.
\]

\item[(iii)] The transversality condition holds:
\begin{equation}\label{eq:global-crossing-D}
\left\langle
\partial_p L_p \big|_{p = p_{\alpha_k}} \hat \Phi_k,\,
\hat \Phi_k
\right\rangle_{L^2(\mathcal C)}
=
- A_k \int_{\mathbb S^{N-1}} \hat Y_k^2(\theta) \, d\theta
< 0,
\end{equation}
where
\[
A_k
=
\frac{m(2m+3k)}{2}
\int_{\mathbb R} \sech^{\frac{2(m+2k)}{k}}(kt) \, dt
> 0.
\]
\end{itemize}
\end{lemma}

\begin{proof}
The proofs of (i) and (ii) are the same as those of Lemma~\ref{KerLk} and Lemma~\ref{lem:codim}, respectively. The proof of (iii) follows as in the proof of \eqref{eq:Dknegative-short}. 
\end{proof}

By \eqref{eq:P-isom}, the restricted operator
\begin{equation}\label{eq:P-isom-g}
\mathcal P \colon \mathcal E_G^2 \to \mathcal E_G^0
\end{equation}
is surjective and invertible. For $p = p_\alpha \in (p_0, \infty)$ with $\alpha > 0$ and $\eta \in \mathcal E_G^0$, define the nonlinear operator
\begin{equation}\label{eq:Kglobal}
\mathcal K(p, \eta)
:=
\mathcal P^{-1}
\left(
(W_{p_\alpha} + \eta)_+^{p_\alpha} - W_{p_\alpha}^{p_\alpha}
\right).
\end{equation}

In view of \eqref{eq:P-isom}, the fact that $p_0 = \frac{N+2}{N-2} > 1$, and the rapid decay of $W_{p_\alpha}$ and its derivatives as $|t| \to \infty$, the proof of the following lemma is straightforward.

\begin{lemma}\label{lem:compactK}
For any closed interval $[a,b] \subset (p_0, \infty)$, the operator
\[
\mathcal K \colon [a,b] \times \mathcal E_G^0 \to \mathcal E_G^0
\]
is compact and of class $C^1$. Moreover, its Fréchet derivative at $\eta = 0$ is given by
\[
D_\eta \mathcal K(p,0)\xi
=
\mathcal P^{-1}
\bigl( p W_p^{p-1} \xi \bigr)
\qquad \text{for all } \xi \in \mathcal E_G^0.
\]
\end{lemma}

Consider the fixed-point equation
\begin{equation} \label{eq:dual-main}
\eta - \mathcal K(p,\eta) = 0
\qquad \text{on } (p_0, \infty) \times \mathcal E_G^0.
\end{equation}
Let \(\eta\not\equiv0\) be a solution of \eqref{eq:dual-main}, and set
\(w:=W_p+\eta\).  If \(w\not\equiv0\), then
\(\mathcal P w=(w_+)^p\ge0\).  By the strong maximum principle, \(w>0\) on
\(\mathcal C\), and hence \(w\) is a positive solution of
\eqref{eq:cylinder-general}.  We say that a function \(f\) on \(\mathcal C\) is
non-radial if \(\nabla_{\mathbb S^{N-1}}f\not\equiv0\).  Since all functions in
\(\mathcal E_G^2\) are even in \(t\), the classification of positive
homoclinic solutions of \eqref{eq:cylinder-general} implies that the only
positive radial fixed point is the trivial one \(w=W_p\), or equivalently
\(\eta=0\).  Besides this trivial branch, the fixed-point equation also contains
the artificial zero branch \(\eta=-W_p\), corresponding to \(w\equiv0\).
Consequently, after excluding these two radial branches, finding non-radial
positive solutions of \eqref{eq:cylinder-general} in \(\mathcal E_G^2\) reduces
to finding non-radial solutions of \eqref{eq:dual-main}.

\begin{lemma}\label{lem:positive-global}
For any compact interval \(J\subset(p_0,\infty)\), there exists $\rho_J > 0$ such that for any $p \in J$, there is no solution $\eta$ of \eqref{eq:dual-main} satisfying
\[
0 < \|W_p + \eta\|_{C^{0,\gamma}_\sigma(\mathcal C)} < \rho_J.
\]
\end{lemma}

\begin{proof}
If $\eta$ is a solution of \eqref{eq:dual-main}, then $w = W_p + \eta$ satisfies
\[
\mathcal P w = w_+^p \quad \text{on } \mathcal C.
\]
By \eqref{eq:P-isom},
\[
\|w\|_{C^{0,\gamma}_\sigma(\mathcal C)}
\le
C \|w_+^p\|_{C^{0,\gamma}_\sigma(\mathcal C)}
\le
C \|w\|_{C^{0,\gamma}_\sigma(\mathcal C)}^p,
\]
where $C > 0$ depends only on $N$, $\sigma$, $\gamma$, and $J$. Since $p \ge p_0 > 1$, this nonlinear inequality cannot hold when $\|w\|_{C^{0,\gamma}_\sigma(\mathcal C)}$ is positive but sufficiently small. The lemma is proved.
\end{proof}

We now reformulate the bifurcation theorem of Gladiali--Grossi--Neves \cite{GGN} as follows.

\begin{theorem}[Cylindrical version of Gladiali--Grossi--Neves]\label{thm:global-cylinder}
Let $k \ge 2$ and set $\alpha_k = 2(k-1)$. Then the following global continua of non-radial solutions of \eqref{eq:dual-main} bifurcate from $(p_{\alpha_k}, 0)$:
\begin{enumerate}[label=\textup{(\roman*)}]
\item If $k$ is odd, there exists at least one continuum of $O(N-1)$-invariant non-radial solutions.
\item If $k$ is even, there exist at least $\lfloor N/2 \rfloor$ continua of non-radial solutions. More precisely, for each
\[
h = 1, \dots, \left\lfloor \frac N2 \right\rfloor,
\qquad
G_h = O(h) \times O(N-h),
\]
there is a continuum of $G_h$-invariant non-radial solutions.
\end{enumerate}
\end{theorem}

\begin{proof}
By Lemma~\ref{lem:compactK}, this is a compact perturbation of the identity. Along the trivial branch $\eta = 0$, its linearization is
\[
I - D_\eta \mathcal K(p,0) = \mathcal P^{-1} L_p,
\]
with $L_p = L_{p_\alpha}$ as in \eqref{lopt}.
By Lemma~\ref{lem:simple-crossing-global}, the kernel at $p = p_{\alpha_k}$ is simple and the transversality condition holds. Therefore Rabinowitz's global bifurcation theorem for compact perturbations of the identity \cite[Theorem 1.3]{Rabinowitz} yields a connected set $\mathcal D_{G,k}$ in the closure of the nontrivial fixed points emanating from $(p_{\alpha_k}, 0)$.

Let
\[
\mathcal S_G^+
=
\left\{
(p,\eta) :
\eta = \mathcal K(p,\eta),\;
W_p + \eta >0,\;
\eta \not\equiv 0
\right\}
\]
and let \(\mathcal C_{G,k}\) be the connected component of
\(\overline{\mathcal S_G^+}\) containing \((p_{\alpha_k},0)\).  By
Lemma~\ref{lem:positive-global}, on every compact interval
\(J\subset(p_0,\infty)\), there exists \(\rho_J>0\) such that any fixed point
with \(p\in J\) and \(W_p+\eta>0\) satisfies
\(\|W_p+\eta\|_{C^{0,\gamma}_\sigma}\ge \rho_J\).  Thus the positive fixed
points in \(\mathcal S_G^+\) are uniformly separated from the artificial zero
branch \(W_p+\eta\equiv0\) on \(J\).  Hence passing from
\(\mathcal D_{G,k}\) to \(\mathcal C_{G,k}\) does not create any new compact
interior alternative.

Finally, we examine the global alternative. Assume that \(\mathcal C_{G,k}\) is bounded and does not meet \(p=p_0\). Then there exist a compact interval $J \subset (p_0, \infty)$ and $R > 0$ such that
\[
\mathcal C_{G,k} \subset J \times \mathcal B_R(0),
\]
where $\mathcal B_R(0) \subset \mathcal E_G^0$ is the open ball centered at $0$ with radius $R$. Since $\mathcal K$ is compact on bounded sets, the closure of $\mathcal C_{G,k}$ is compact. Rabinowitz's alternative therefore forces the closure to meet the trivial branch at some second point $(p_{\alpha_*}, 0)$ with $p_{\alpha_*} \neq p_{\alpha_k}$.

We now show that $L_{p_{\alpha_*}}$ has nontrivial kernel in the same symmetry class. Choose $(p_n, \eta_n) \in \mathcal C_{G,k}$ with $(p_n, \eta_n) \to (p_{\alpha_*}, 0)$ and $\eta_n \neq 0$, and set
\[
\xi_n = \frac{\eta_n}{\|\eta_n\|_{C^{0,\gamma}_\sigma(\mathcal C)}}.
\]
Using the Fréchet expansion of $\mathcal K$ at $(p_{\alpha_*}, 0)$ from Lemma~\ref{lem:compactK}, we obtain
\[
\eta_n = D_\eta \mathcal K(p_{\alpha_*}, 0)\eta_n
+ o(\|\eta_n\|_{C^{0,\gamma}_\sigma(\mathcal C)}).
\]
Dividing by $\|\eta_n\|_{C^{0,\gamma}_\sigma(\mathcal C)}$ gives
\[
\xi_n
=
D_\eta \mathcal K(p_{\alpha_*}, 0)\xi_n + o(1)
=
\mathcal P^{-1}
\bigl(
p_{\alpha_*} W_{p_{\alpha_*}}^{p_{\alpha_*}-1} \xi_n
\bigr)
+ o(1)
\qquad \text{in } \mathcal E_G^0.
\]
Since $\mathcal P^{-1}(p_{\alpha_*} W_{p_{\alpha_*}}^{p_{\alpha_*}-1} \cdot)$ is compact, a subsequence converges to some $\xi \neq 0$, and passing to the limit yields
\[
\xi = \mathcal P^{-1}
\bigl(
p_{\alpha_*} W_{p_{\alpha_*}}^{p_{\alpha_*}-1} \xi
\bigr).
\]
Equivalently,
\[
L_{p_{\alpha_*}} \xi = 0.
\]
By Theorem~\ref{thm:kernel}, this can happen only when
\[
p_{\alpha_*} = 1 + \frac{2j}{m}
\]
for some integer $j \neq k$ such that the degree-$j$ harmonic space contains a nonzero $G$-invariant vector. In terms of $\alpha$, this means
\[
\alpha_* = m(p_{\alpha_*} - 1) - 2 = 2(j-1) = \alpha_j.
\]
Thus the closure of the component either is unbounded, meets $p = p_0$ (equivalently $\alpha = 0$), or meets another bifurcation point $(\alpha_j, 0)$ in the same symmetry class.
\end{proof}

\begin{remark}\label{gbilm}
Note that the last part of the proof verifies something more than the statement of the theorem: it shows that any bounded continuum must either meet $p=p_0$ or reconnect to another even bifurcation point $(\alpha_j,0)$ in the same symmetry class.
\end{remark}

\section*{Declaration}

\noindent\textbf{Data availability:} Data availability is not applicable to this article as no new data were created or analyzed in this study.

\medskip

\noindent\textbf{Conflict of interest:} The authors declare that they have no conflicts of interests.


\begin{thebibliography}{99}
\bibitem{AAR}
G. E. Andrews, R. Askey and R. Roy,
\emph{Special Functions},
Encyclopedia of Mathematics and its Applications, vol.~71,
Cambridge University Press, Cambridge, 1999.

\bibitem{BadialeSerra}
M. Badiale and E. Serra,
Multiplicity results for the supercritical H\'enon equation,
\emph{Adv. Nonlinear Stud.} \textbf{4} (2004), no.~4, 453--467.

\bibitem{BoscagginColasuonnoNorisWeth2023}
A. Boscaggin, F. Colasuonno, B. Noris and T. Weth,
A supercritical elliptic equation in the annulus,
\emph{Ann. Inst. H. Poincar\'e C Anal. Non Lin\'eaire} \textbf{40} (2023),
no.~1, 157--183.

\bibitem{CGS}
L. Caffarelli, B. Gidas and J. Spruck,
Asymptotic symmetry and local behavior of semilinear elliptic equations with
critical Sobolev growth,
\emph{Comm. Pure Appl. Math.} \textbf{42} (1989), 271--297.

\bibitem{Conway}
J. B. Conway,
\emph{A Course in Functional Analysis},
2nd edn., Graduate Texts in Mathematics, vol.~96,
Springer, New York, 1990.

\bibitem{CowanMoameni2024}
C. Cowan and A. Moameni,
On supercritical elliptic problems: existence, multiplicity of positive and
symmetry breaking solutions,
\emph{Math. Ann.} \textbf{389} (2024), no.~2, 1731--1794.

\bibitem{CR}
M. G. Crandall and P. H. Rabinowitz,
Bifurcation from simple eigenvalues,
\emph{J. Functional Analysis} \textbf{8} (1971), 321--340.

\bibitem{DaiDuanGuiLi2026}
W. Dai, L. Duan, C. Gui and Y. Li,
Non-radial solutions for the critical quasi-linear H\'enon equation involving
\(p\)-Laplacian in \(\mathbb R^N\),
\emph{Proc. London Math. Soc.} \textbf{132} (2026), no.~4, Paper No.~e70148.

\bibitem{FigueroaNeves}
P. Figueroa and S. L. N. Neves,
Nonradial solutions for the H\'enon equation close to the threshold,
\emph{Adv. Nonlinear Stud.} \textbf{19} (2019), no.~4, 757--770.

\bibitem{Gasper}
G. Gasper,
Linearization of the product of Jacobi polynomials. I,
\emph{Canad. J. Math.} \textbf{22} (1970), 171--175.

\bibitem{GGN}
F. Gladiali, M. Grossi and S. L. N. Neves,
Nonradial solutions for the H\'enon equation in \(\mathbb R^N\),
\emph{Adv. Math.} \textbf{249} (2013), 1--36.

\bibitem{HanXiongZhang2023} 
Z. C. Han, J. Xiong and L. Zhang, Asymptotic behavior of solutions to the Yamabe equation with an asymptotically flat metric,
\emph{J. Funct. Anal.} \textbf{285} (2023), no.~11, Paper No.~109982.

\bibitem{Kato}
T. Kato,
\emph{Perturbation Theory for Linear Operators},
2nd edn., Classics in Mathematics,
Springer, Berlin, 1995.

\bibitem{LM}
R. B. Lockhart and R. C. McOwen,
Elliptic differential operators on noncompact manifolds,
\emph{Ann. Scuola Norm. Sup. Pisa Cl. Sci. (4)} \textbf{12} (1985), 409--447.


\bibitem{Marques2008} F. C. Marques, Isolated singularities of solutions to the Yamabe equation, \emph{Calc. Var. Partial Differential Equations} \textbf{32} (2008), 349--371.

\bibitem{P}
F. Pacard,
\emph{Connected sum constructions in geometry and nonlinear analysis},
lecture notes, 2008.

\bibitem{Perko}
L. Perko,
\emph{Differential Equations and Dynamical Systems},
3rd edn., Texts in Applied Mathematics, vol.~7,
Springer, New York, 2001.

\bibitem{PT}
J. Prajapat and G. Tarantello,
On a class of elliptic problems in \(\mathbb R^2\): symmetry and uniqueness results,
\emph{Proc. Roy. Soc. Edinburgh Sect. A} \textbf{131} (2001), 967--985.

\bibitem{Rabinowitz}
P. H. Rabinowitz,
Some global results for nonlinear eigenvalue problems,
\emph{J. Functional Analysis} \textbf{7} (1971), 487--513.

\bibitem{SW71}
E. M. Stein and G. Weiss,
\emph{Introduction to Fourier Analysis on Euclidean Spaces},
Princeton University Press, Princeton, 1971.

\bibitem{Teschl}
G. Teschl,
\emph{Mathematical Methods in Quantum Mechanics},
Graduate Studies in Mathematics, vol.~157,
American Mathematical Society, Providence, RI, 2014.

\bibitem{XiongZhang2022} 
J. Xiong and L. Zhang, Isolated singularities of solutions to the Yamabe equation in dimension 6, 
\emph{Int. Math. Res. Not. IMRN} \textbf{2022} (2022), no.~12, 9571--9597.


\end{thebibliography}
\end{document}